\documentclass[12pt]{article}
\usepackage{amssymb}
\usepackage{hyperref}
\usepackage{latexsym}
\usepackage{color}
\usepackage{amsmath}
\usepackage{graphicx}
\usepackage{amsfonts}
\usepackage{amsthm}
\usepackage{graphicx}
\usepackage{mathrsfs}
\newtheorem{theorem}{Theorem}
\newtheorem{lemma}{Lemma}

\everymath{\displaystyle}

\begin{document}

\newcommand{\al}{\alpha}
\newcommand{\bt}{\beta}
\newcommand{\ep}{\varepsilon}
\newcommand{\dl}{\delta}
\newcommand{\ti}{\tilde}
\newcommand{\be}{\begin{equation}}
\newcommand{\ee}{\end{equation}}
\newcommand{\goto}{\rightarrow}
\newcommand{\half}{\frac{1}{2}}
\newcommand{\ld}{\lambda}
\newcommand{\p}{\partial}
\newcommand{\ph}{\varphi}
\newcommand{\ov}{\overline}
\newcommand{\pl}{\parallel}
\newcommand{\br}{\breve}
\newcommand{\var}{\varphi}

\newcommand\Ker{{\rm Ker}}
\newcommand\intl{\int\limits}
\newcommand\suml{\sum\limits}
\newcommand\maxl{\max\limits}
\newcommand\minl{\min\limits}
\newcommand\supl{\sup\limits}
\newcommand\infl{\inf\limits}
\newcommand\liml{\lim\limits}
\newcommand\ch{{\rm ch}}
\newcommand\tg{{\rm tg}}
\newcommand\rank{{\rm rank}}
\newcommand\const{{\rm const}}
\newcommand\Ga{\Gamma}
\renewcommand{\Im}{{\rm Im}}
\makeatletter \@addtoreset{equation}{section}
\renewcommand{\theequation}{\thesection.\arabic{equation}}
\makeatother

\font\bfsl=cmssi17 \font\sectiontt=cmtt12 scaled\magstep1
\font\stt=cmtt10 scaled 850

\renewcommand{\baselinestretch}{1}
\renewcommand{\theequation}{\arabic{section}.\arabic{equation}}
\renewcommand{\thetheorem}{\arabic{section}.\arabic{theorem}}
\renewcommand{\thelemma}{\arabic{section}.\arabic{lemma}}
\renewcommand{\theproposition}{\arabic{section}.\arabic{proposition}}
\renewcommand{\thedefinition}{\arabic{section}.\arabic{definition}}
\renewcommand{\thecorollary}{\arabic{section}.\arabic{corollary}}
\renewcommand{\theremark}{\arabic{section}.\arabic{remark}}
\pagestyle{empty}

 \pagestyle{myheadings}

 \pagestyle{myheadings} \pagestyle{myheadings}
\title{{\bf{Upper Bound For The  Ratios Of Eigenvalues Of Schrodinger
Operators With Nonnegative Single-Barrier Potentials}}}
\author{Jamel Ben Amara  \thanks{Faculty of Sciences of Tunis,
University of Tunis El Manar, Mathematical Engineering Laboratory,
Polytechnic School of Tunisia,
jamel.benamara@fsb.rnu.tn}~~~~~~~~~~Jihed Hedhly
\thanks{Faculty of Sciences  Bizerte, University of Carthage,
  Mathematical Engineering Laboratory,
Polytechnic School of Tunisia, hjihed@gmail.com.}\thanks{Accepted in
Mathematische Nachrichten 2018.}}

\date{}
\date{}
\date{}

 \maketitle
\begin{abstract}
In this paper we prove the optimal upper bound
$\frac{\lambda_{n}}{\lambda_{m}}\leq\frac{n^{2}}{m^{2}}$
$\Big(\lambda_{n}>\lambda _{m}\geq
11\sup\limits_{x\in[0,1]}q(x)\Big)$ for one-dimensional
 Schrodinger operators with a nonnegative
differentiable and single-barrier potential $q(x)$, such that $\mid
q'(x) \mid\leq q^{*},$ where $q^{*}=\frac{2}{15}\min\{q(0) ,
q(1)\}$. In particular, if $q(x)$ satisfies the additional condition
$\sup\limits_{x\in[0,1]}q(x)\leq \frac{\pi^{2}}{11}$, then
$\frac{\lambda _{n}}{\lambda _{m}}\leq \frac{n^{2}%
}{m^{2}}$ for $n>m\geq 1.$ For this result, we develop a new
approach to study the monotonicity of the modified Pr\"{u}fer angle
function.
\end{abstract}
~~~~~~~~~~\\
{\it{ 2000 Mathematics Subject Classification. Primary 34L15, 34B24. \\
Key words and phrases. Sturm-Liouville Problems, eigenvalue ratio,
single-barrier, Pr\"{u}fer substitution}.}
\section{Introduction}
Consider the Sturm-Liouville equation acting on $[0,1]$
\begin{equation} \label{1.1}
-y^{\prime \prime }+q(x)y=\lambda y,
\end{equation}%
with Dirichlet boundary conditions
\begin{equation}
y(0)=y(1)=0.  \label{1.2}
\end{equation}%
Here $q$ is a nonnegative differentiable and single-barrier
potential in $[0,1]$.\\ A function in $[0, 1]$ is called single-well
(single-barrier) if there is $x_0$ point a in $[0, 1]$, such that it
is monotone decreasing (increasing) in $[0,x_0]$ and monotone
increasing (decreasing) in $[x_0,1]$. The point $x_0$ is the
transition point.

It is known (see ~ \cite{5}) that the spectrum of Problem $
\eqref{1.1} - \eqref{1.2}$ consists of a growing sequence of
infinitely point $\lambda _{1}<\lambda _{2}<....<\lambda
_{n}...\infty$. Ashbaugh and Benguria \cite{1',1} proved the optimal
bound $\frac{\lambda_{n}}{\lambda_{1}}\leq n^{2},$ for nonnegative
potentials. They also established the ratio estimate  $
\frac{\lambda _{n}}{\lambda _{m}}\leq ( \lceil \frac{n}{m} \rceil )
^{2},$ for $n> m\geq1,$ where $\lceil s\rceil $ denotes the smallest
integer greater than or equal to $s$. Later, Huang and Law in
\cite{CK} proved that the eigenvalues of the regular Sturm-Liouville
equation $-(p(x)y^{\prime})^ {\prime }+q(x)y=\lambda \rho (x)y$
(with Dirichlet boundary conditions) satisfy the lower bound
$$\frac{\lambda_{n}}{\lambda_{m}}\geq\frac{1}{1+\xi}
\Big(\frac{n+1}{m+1}\Big)^2\frac{k}{K},~~n>m\geq0,$$ for $p\in
C^1[0,1],$ $q,\rho\in C[0,1]$, $q\geq0$ and $0<k\leq p(x)\rho(x)\leq
K,$ with $\xi=\frac{K\max{\{pq\}}}{kn^2\sigma^2\pi^2}$ and
$\sigma=\Big(\int_{0}^{1}\frac{1}{p(s)}ds\Big)^{-1}$. In $2005$,
Horv\'{a}th and Kiss \cite{3} showed that
\begin{eqnarray}\label{ratios} \frac{\lambda _{n}}{\lambda _{m}}\leq
\frac{n^{2}}{m^{2}},~~n>m\geq1,\end{eqnarray} for nonnegative
single-well potentials. Their approach is mainly based on the
monotonicity of the Pr\"{u}fer angle as function in $\lambda\geq0$
(see \cite[~Theorem ~2.2]{3}). At the end of their paper
\cite[~Remark~5.1]{3}, they gave an example of a single-barrier
potential which shows that the associated Pr\"{u}fer angle is not a
monotonous function.

 In the present paper, we give additional
conditions on the single-barrier potential $q(x)$ for which Theorem
2.2 in \cite{3} and the estimate \eqref{ratios} remain valid.
Namely,  we prove that if $q(x)$ is a nonnegative and single-barrier
potential (with transition point $x_{0}\in[0,1]$), such that
$|q'(x)|\leq q^{*}$ where $q^{*}=\frac{2}{15}\min\{q(0),~q(1)\}$,
then $\frac{\lambda _{n}}{\lambda _{m}}\leq \frac{n^{2}}{m^{2}}$ for
$\lambda _{n}>\lambda _{m}\geq 11q(x_{0})$. In particular, if $q(x)$
satisfies the additional condition $q(x_{0})\leq
\frac{\pi^{2}}{11}$, then the last bound estimate holds for $n>m\geq
1.$ Note that, our approach used in this paper can be applied to the
case of nonnegative and single-well potentials studied in \cite{3}
(without further restrictions on $q(x)$).


\section{Preliminaries And The Main Statement} \label{ps}
\label{required} Following \cite{3}, we introduce the modified
Pr\"{u}fer transformation.\\ Denote by $y( x,z ) $
 the unique solution of the initial value problem%
\begin{equation}\label{H.S}
-y^{\prime \prime }+q(x)y=z^{2}y,\ \ \ x\in [ 0,1] ,\ \ \ z>0,
\end{equation}%
\begin{equation}
y(0)=0, \ \ \  y^{\prime }(0)=1.  \label{J.R}
\end{equation}
The Pr\"{u}fer variables $r( x,z) ,$ $\varphi ( x,z) $ that we use
here are defined by
\begin{eqnarray}\label{2.2}
y(x,z) =r( x,z) \sin \varphi ( x,z),
\end{eqnarray}
\begin{eqnarray}\label{2.2a}
y^{\prime }( x,z) =zr( x,z) \cos \varphi ( x,z),
\end{eqnarray}
\begin{eqnarray}\label{2.2b}
\varphi ( 0,z) =0,
\end{eqnarray}
where $r(x,z)>0$, and then let $\theta(x,z) =\frac{\varphi(x,z) }{z}.$\\
We denote by prime (resp. dot) the derivative with respect to $x$
(resp. $z$). \\ Using Equation \eqref{1.1}  one finds the following
differential equations for $r( x,z) $ and $\varphi(x,z):$
\begin{equation} \label{2.6}
\varphi ^{\prime }=z-\frac{q}{z}\sin ^{2}\varphi ,
\end{equation} and
\begin{equation}\label{2.7}
\frac{r^{\prime }}{r}=\frac{q}{z}\sin \varphi \cos \varphi .
\end{equation}
It is obvious that $z^{2}$ is an eigenvalue iff $\varphi(\pi ,z) $
is a multiple of $\pi $. Denote by $z_{n}$ the square root of
$\lambda _{n}$ $( \lambda _{n}=z_{n}^{2}) $. Moreover by
\eqref{2.6}, $\varphi ^{\prime }(x,z)>0$ for $x\in[0,x_{0}]$ and
$z^2>q(x_{0}).$ In this case $\varphi ^{-1}$ exists and $\varphi
^{-1}(k\pi+\frac{\pi}{2})$, $\varphi ^{-1}((k+1)\pi)$ ($k\geq 0$ be
an integer) are the zeros of $y'$ and $y$ in $(0,x_{0}]$,
respectively. It is known (e.g., see \cite[~chap.1]{5}) that these
zeros are decreasing as $z$ increases. We now enunciate the main
result of this paper.

\begin{theorem}\label{the2}
For the Sturm-Liouville problem \eqref{1.1}-\eqref{1.2}, if $q(x)$
is a nonnegative differentiable and single-barrier potential with
transition point $x_{0}\in[0,1]$ such that $|q'(x)|\leq q^{*}$ where
$q^{*}=\frac{2}{15}\min\{q(0),~q(1)\}$, then
\begin{equation}
\frac{\lambda _{n}}{\lambda _{m}}\leq \frac{n^{2}}{m^{2}},~~for~
\lambda _{n}>\lambda _{m}\geq 11q(x_{0}).
\end{equation}
In particular, if $q(x)$ satisfies the additional condition
$q(x_{0})\leq \frac{\pi^{2}}{11},$ then $\frac{\lambda _{n}}{\lambda
_{m}}\leq \frac{n^{2} }{m^{2}},$ for \\$n>m\geq 1.$
\\If for two different $m$ and $n$ equality holds, then $q\equiv 0 $ in $[0,1].$
\end{theorem}
The proof of Theorem \ref{the2} will be given in Section \ref{4}.


\section{Monotonicity Of The Pr\"{u}fer Angle Function $\theta(x_{0},z).$}
 In this section, we study the monotonicity
of the Pr\"{u}fer angle function $\theta(x_{0},z).$
\begin{theorem}\label{the1}
Let $q(x)\geq 0$ be monotone increasing and differentiable in $[
0,x_{0}]$ such that $q'(x)\leq \frac{2}{15}q(0)$. Then
$\dot{\theta}( x_{0},z) \geq 0$ for\ $z\geq \sqrt{11q(x_{0})}.$ If
there is a $z\geq \sqrt{11q(x_{0})}$ with $\dot{\theta}( x_{0},z)
=0,$ then $q\equiv 0$ in $[0,x_{0}]$.
\end{theorem}
To prove this theorem we need some preliminary results.
\begin{lemma} \label{cor1} (Corollary~ 3.3 in \cite{3})
\begin{equation}\label{2.9}
\dot{\theta}( x,z) =\frac{2}{z^{2}r^{2}(x)}\int_{0}^{x}r^{2}(t)%
\frac{q(t)}{z}\Big( \sin ^{2}\varphi -\varphi \sin \varphi \cos
\varphi \Big) dt.
\end{equation}
\end{lemma}
The following result plays a fundamental role in the sequel.
\begin{lemma}\label{Lem4} Let $i\geq0$ be an integer and assume
that $\varphi ^{-1}( i\pi +\frac{\pi }{2}+D)\in(0,x_{0}],$ where
$0\leq D\leq\pi.$ Then
\begin{eqnarray}\label{novo}
&&\int_{\varphi ^{-1}( i\pi +\frac{\pi }{2}) }^{\varphi ^{-1}(i\pi
+\frac{\pi }{2}+D)}r^{2}(x) \frac{q(x)}{z}\Big( \sin^{2}\varphi
-\varphi\sin\varphi\cos\varphi\Big) dx \cr &&\geq r^{2}(\varphi
^{-1}(i\pi +\frac{\pi }{2}+D))\Big[ \int_{\varphi ^{-1}( i\pi
+\frac{\pi }{2}) }^{\varphi ^{-1}(i\pi +\frac{\pi
}{2}+D)}\frac{q(x)}{z}\Big( \sin^{2}\varphi
-\varphi\sin\varphi\cos\varphi\Big) dx \cr &&-2\int_{\varphi ^{-1}(
i\pi +\frac{\pi }{2}) }^{\varphi ^{-1}( i\pi +\frac{\pi
}{2}+D)}\frac{q(x)}{z}\sin\varphi\cos\varphi
\Big(\int_{\varphi^{-1}(i\pi+\frac{\pi}{2})}^{x}\frac{q(s)}{z}
\sin^{2}\varphi ds \Big)dx \Big].
\end{eqnarray}
\end{lemma}
\begin{proof} By \eqref{2.7},
\begin{eqnarray*}
&&- \int_{\varphi ^{-1}(i\pi +\frac{\pi }{2}) }^{\varphi ^{-1} (i\pi
+\frac{\pi }{2}+D)}{\textstyle r^{2}(x) \frac{q(x)}{z}\varphi
\sin\varphi\cos\varphi dx}\cr&&=-\int_{\varphi ^{-1}( i\pi
+\frac{\pi }{2}) }^{\varphi ^{-1}(i\pi +\frac{\pi }{2}+D)}r(x)r'(x)
\varphi dx\cr&=&-\frac{1}{2}\Big[r^{2}(x)\varphi(x)\Big]_{\varphi
^{-1}( i\pi +\frac{\pi }{2}) }^{\varphi ^{-1}(i\pi +\frac{\pi
}{2}+D)}+\frac{1}{2} \int_{\varphi ^{-1}( i\pi +\frac{\pi }{2})
}^{\varphi ^{-1}(i\pi +\frac{\pi }{2}+D)}{\textstyle
r^{2}(x)\varphi'(x)dx}.
\end{eqnarray*}
Since $r^{2}(x)=e^{2\int_{0}^{x}\frac{q(s)}{z}\sin\varphi\cos\varphi
ds}$, then
\begin{eqnarray}\label{th}
&&- \int_{\varphi ^{-1}(i\pi +\frac{\pi }{2}) }^{\varphi ^{-1} (i\pi
+\frac{\pi }{2}+D)}{\textstyle r^{2}(x) \frac{q(x)}{z}\varphi
\sin\varphi\cos\varphi dx} \cr &=&-\frac{1}{2}\Big[(i\pi +\frac{\pi
}{2}+D)r^{2}(\varphi^{-1}(i\pi +\frac{\pi
}{2}+D)-(i\pi+\frac{\pi}{2})r^{2}(\varphi^{-1}({i\pi}+\frac{\pi}{2}))\Big]
\cr&&+\frac{1}{2}\int_{\varphi ^{-1}(i\pi +\frac{\pi }{2})
}^{\varphi ^{-1}(i\pi +\frac{\pi }{2}+D)} r^{2}(x)\varphi'(x) dx
\cr&=&-\frac{1}{2}r^{2}(\varphi^{-1}(i\pi +\frac{\pi
}{2}+D)\Big[D-\int_{\varphi ^{-1}(i\pi +\frac{\pi }{2}) }^{\varphi
^{-1}(i\pi +\frac{\pi }{2}+D)}
\frac{r^{2}(x)}{r^{2}(\varphi^{-1}(i\pi +\frac{\pi
}{2}+D))}\varphi'(x)
dx\Big]\cr&&-(\frac{i\pi+\frac{\pi}{2}}{2})r^{2}(\varphi^{-1}(i\pi
+\frac{\pi }{2}+D)) \Big[1-\frac{r^{2}(\varphi ^{-1}( i\pi
+\frac{\pi }{2}))}{r^{2}(\varphi^{-1}(i\pi +\frac{\pi }{2}+D))}\Big]
\cr&=&-\frac{1}{2}r^{2}(\varphi^{-1}(i\pi +\frac{\pi
}{2}+D)\Big[D-\int_{\varphi ^{-1}( i\pi +\frac{\pi }{2}) }^{\varphi
^{-1}(i\pi +\frac{\pi }{2}+D)} \varphi'(x)g(x,z) dx\Big]
\cr&-&(\frac{i\pi+\frac{\pi}{2}}{2})r^{2}(\varphi^{-1}(i\pi
+\frac{\pi }{2}+D)) \Big[1-g(\varphi ^{-1}( i\pi +\frac{\pi
}{2}),z)\Big],
\end{eqnarray}
where $g(x,z)=e^{-2\int_{x}^{\varphi ^{-1}(i\pi +\frac{\pi }{2}+D)}
\frac{q(s)}{z}\sin\varphi\cos\varphi ds}$. \\Using the inequality
\begin{eqnarray*}&&g(x,z)=e^{-2\int_{x}^{\varphi ^{-1}(i\pi
+\frac{\pi }{2}+D)} \frac{q(s)}{z}\sin\varphi\cos\varphi
ds}\cr&&\geq1-2\int_{x}^{\varphi ^{-1}(i\pi +\frac{\pi }{2}+D)}
\frac{q(s)}{z}\sin\varphi\cos\varphi ds,\end{eqnarray*} we obtain
\begin{eqnarray}\label{r1}
&&-\frac{1}{2}r^{2}(\varphi^{-1}(i\pi +\frac{\pi
}{2}+D)\Big[D-\int_{\varphi ^{-1}( i\pi +\frac{\pi }{2}) }^{\varphi
^{-1}(i\pi +\frac{\pi }{2}+D)} \varphi'(x)g(x,z) dx\Big]\cr&&\geq
-r^{2}(\varphi ^{-1}(i\pi +\frac{\pi }{2}+D))\int_{\varphi ^{-1}(
i\pi +\frac{\pi }{2}) }^{\varphi ^{-1}(i\pi +\frac{\pi }{2}+D)}
{\textstyle\varphi'(x)\Big(\int_{x}^{\varphi ^{-1}(i\pi +\frac{\pi
}{2}+D)}\frac{q(s)}{z}\sin\varphi\cos\varphi ds\Big)
dx}\cr&&=r^{2}(\varphi ^{-1}(i\pi +\frac{\pi }{2}+D))\Big(
-\Big[\varphi(x)\int_{x}^{\varphi ^{-1}(i\pi +\frac{\pi
}{2}+D)}\frac{q(s)}{z}\sin\varphi\cos\varphi ds\Big]_{\varphi ^{-1}(
i\pi +\frac{\pi }{2}) }^{\varphi ^{-1}(i\pi +\frac{\pi
}{2}+D)}\cr&&-\int_{\varphi ^{-1}( i\pi +\frac{\pi }{2}) }^{\varphi
^{-1}(i\pi +\frac{\pi }{2}+D)}
\frac{q(x)}{z}\varphi(x)\sin\varphi\cos\varphi
dx\Big)\cr&=&r^{2}(\varphi ^{-1}(i\pi +\frac{\pi
}{2}+D))\Big((i\pi+\frac{\pi}{2}) \int_{\varphi ^{-1}( i\pi
+\frac{\pi }{2}) }^{\varphi ^{-1}(i\pi +\frac{\pi
}{2}+D)}\frac{q(x)}{z}\sin\varphi\cos\varphi dx\cr&&-\int_{\varphi
^{-1}( i\pi +\frac{\pi }{2}) }^{\varphi ^{-1}(i\pi +\frac{\pi
}{2}+D)} \frac{q(x)}{z}\varphi(x)\sin\varphi\cos\varphi dx.
\end{eqnarray}
Similarly,
\begin{eqnarray}\label{g}
 &&-(\frac{i\pi+\frac{\pi}{2}}{2})r^{2}(\varphi^{-1}(i\pi
+\frac{\pi }{2}+D)\Big[1-g(\varphi ^{-1}( i\pi +\frac{\pi
}{2}),z)\Big]\cr&&\geq
(\frac{i\pi+\frac{\pi}{2}}{2})r^{2}(\varphi^{-1}(i\pi +\frac{\pi
}{2}+D)\int_{\varphi ^{-1}( i\pi +\frac{\pi }{2}) }^{\varphi
^{-1}(i\pi +\frac{\pi }{2}+D)}\frac{q(x)}{z}\sin\varphi\cos\varphi
dx.
\end{eqnarray}
Using \eqref{r1} and \eqref{g},
\begin{eqnarray}\label{ge}
&&- \int_{\varphi ^{-1}(i\pi +\frac{\pi }{2}) }^{\varphi ^{-1} (i\pi
+\frac{\pi }{2}+D)}{\textstyle r^{2}(x) \frac{q(x)}{z}\varphi
\sin\varphi\cos\varphi dx} \cr&&\geq -r^{2}(\varphi^{-1}(i\pi
+\frac{\pi }{2}+D)\int_{\varphi ^{-1}( i\pi +\frac{\pi }{2})
}^{\varphi ^{-1}(i\pi +\frac{\pi }{2}+D)}
\frac{q(x)}{z}\varphi(x)\sin\varphi\cos\varphi dx.
\end{eqnarray}
 On the other hand,
\begin{eqnarray*}
&&\int_{\varphi ^{-1}( i\pi +\frac{\pi }{2}) }^{\varphi ^{-1}(i\pi
+\frac{\pi }{2}+D)}r^{2}(x) \frac{q(x)}{z} \sin^{2}\varphi dx \cr&&=
r^{2}(\varphi ^{-1}(i\pi +\frac{\pi }{2}+D)) \int_{\varphi ^{-1}(
i\pi +\frac{\pi }{2}) }^{\varphi ^{-1}(i\pi +\frac{\pi
}{2}+D)}\frac{q(x)}{z}\sin^{2}\varphi(x) g(x,z)dx\cr &&\geq
r^{2}(\varphi ^{-1}(i\pi +\frac{\pi }{2}+D))\Big[ \int_{\varphi
^{-1}( i\pi +\frac{\pi }{2}) }^{\varphi ^{-1}(i\pi +\frac{\pi
}{2}+D)}{\textstyle\frac{q(x)}{z}\sin^{2}\varphi dx }\cr \cr&&-2
\int_{\varphi ^{-1}( i\pi +\frac{\pi }{2}) }^{\varphi ^{-1}(i\pi
+\frac{\pi }{2}+D)}{\textstyle \frac{q(x)}{z}\sin^{2}\varphi
\Big(\int_{x}^{\varphi^{-1}(i\pi +\frac{\pi }{2}+D)}\frac{q(s)}{z}
\sin\varphi\cos\varphi ds \Big)dx}.
\end{eqnarray*}
Integrating by parts, yields
\begin{eqnarray*}
&&-2 \int_{\varphi ^{-1}( i\pi +\frac{\pi }{2}) }^{\varphi
^{-1}(i\pi +\frac{\pi }{2}+D)}{\textstyle
\frac{q(x)}{z}\sin^{2}\varphi \Big(\int_{x}^{\varphi^{-1}(i\pi
+\frac{\pi }{2}+D)}\frac{q(s)}{z} \sin\varphi\cos\varphi ds
\Big)dx}\cr&&={\textstyle-2\Big[\int_{\varphi ^{-1}( i\pi +\frac{\pi
}{2}) }^{x}\frac{q(s)}{z}\sin^{2}\varphi ds\int_{x}^{\varphi
^{-1}(i\pi +\frac{\pi }{2}+D)}\frac{q(s)}{z}\sin\varphi\cos\varphi
ds\Big]}_{\varphi ^{-1}( i\pi +\frac{\pi }{2}) }^{\varphi ^{-1}(i\pi
+\frac{\pi }{2}+D)} \cr \cr&&-2 \int_{\varphi ^{-1}( i\pi +\frac{\pi
}{2}) }^{\varphi ^{-1}(i\pi +\frac{\pi
}{2}+D)}{\textstyle\frac{q(x)}{z} \sin\varphi\cos\varphi
 \Big(\int_{\varphi
^{-1}( i\pi +\frac{\pi }{2}) }^{x}\frac{q(s)}{z}\sin^{2}\varphi ds
\Big)dx}\cr&&=-2 \int_{\varphi ^{-1}( i\pi +\frac{\pi }{2})
}^{\varphi ^{-1}(i\pi +\frac{\pi }{2}+D)}{\textstyle\frac{q(x)}{z}
\sin\varphi\cos\varphi
 \Big(\int_{\varphi^{-1}( i\pi +\frac{\pi }{2})}^{x}\frac{q(s)}{z}\sin^{2}\varphi ds
\Big)dx.}
\end{eqnarray*}
Therefore,
\begin{eqnarray}\label{he}
&&\int_{\varphi ^{-1}( i\pi +\frac{\pi }{2}) }^{\varphi ^{-1}(i\pi
+\frac{\pi }{2}+D)}r^{2}(x) \frac{q(x)}{z} \sin^{2}\varphi dx
\cr&&\geq r^{2}(\varphi ^{-1}(i\pi +\frac{\pi
}{2}+D))\Big[\int_{\varphi ^{-1}( i\pi +\frac{\pi }{2}) }^{\varphi
^{-1}(i\pi +\frac{\pi
}{2}+D)}{\textstyle\frac{q(x)}{z}\sin^{2}\varphi dx }\cr \cr&&-2
\int_{\varphi ^{-1}( i\pi +\frac{\pi }{2}) }^{\varphi ^{-1}(i\pi
+\frac{\pi }{2}+D)}{\textstyle\frac{q(x)}{z} \sin\varphi\cos\varphi
 \Big(\int_{\varphi
^{-1}( i\pi +\frac{\pi }{2}) }^{x}\frac{q(s)}{z}\sin^{2}\varphi ds
\Big)dx}.
\end{eqnarray}

 From \eqref{ge} and \eqref{he}, we get \eqref{novo}.
 This completes the proof of the lemma.
\end{proof}
Using the substitution $t=\varphi(x)-(i+1)\pi$ $(i\geq0$), we obtain
the following result:
\begin{lemma}\label{ty} Assume that $\varphi ^{-1}(i\pi
+\frac{\pi }{2}+D)\in(0,x_{0}]$
 $(i\geq0,~0\leq D\leq\pi),$ then
\begin{eqnarray} \label{jh}
&&\int_{\varphi ^{-1}( i\pi +\frac{\pi }{2}) }^{\varphi ^{-1}( i\pi
+\frac{\pi }{2}+D)}\frac{q(x)}{z}\Big(\sin^{2}\varphi-\varphi\sin
\varphi \cos \varphi \Big)dx\cr&&=\int_{-\frac{\pi }{2}}^{-\frac{\pi
}{2}+D}\frac{\frac{q(\varphi ^{-1}(t+(i+1)\pi))}{z^{2}}
\Big(\sin^{2}t-t\sin t\cos t\Big)}{1-\frac{q(\varphi
^{-1}(t+(i+1)\pi))}{z^{2}}\sin^{2}t}dt\cr&&-(i+1)\pi\int_{-\frac{\pi
}{2}}^{-\frac{\pi }{2}+D} \frac{\frac{q(\varphi
^{-1}(t+(i+1)\pi))}{z^{2}} \sin t\cos t}{1-\frac{q(\varphi
^{-1}(t+(i+1)\pi))}{z^{2}}\sin^{2}t}dt.
\end{eqnarray}
\end{lemma}
\begin{lemma} \label{wa} Let $q(x)\geq 0$ be monotone increasing
on $[0,x_{0}]$ such that $z^{2}> q(x_{0})$ and $\varphi ^{-1}(i\pi
+\frac{\pi }{2}+D)\in(0,x_{0}]$ $(i\geq0)$ with $\frac{\pi}{2}\leq
D\leq\pi.$ Then
\begin{description}
  \item[i)] \begin{eqnarray}\label{aa}
{\textstyle\frac{\pi}{4}\Big(\frac{\frac{q(\varphi ^{-1}( i\pi
+\frac{\pi }{2}))}{z^{2}}}{1-\frac{3q(\varphi ^{-1}( i\pi +\frac{\pi
}{2}))}{4z^{2}}}\Big)}\leq\int_{\varphi ^{-1}( i\pi +\frac{\pi }{2})
}^{\varphi ^{-1}(i\pi +\frac{\pi
}{2}+D)}\frac{q(x)}{z}\sin^{2}\varphi dx\leq{\textstyle
\frac{\pi}{2}\Big(\frac{\frac{q(\varphi ^{-1}(i\pi +\frac{\pi
}{2}+D))}{z^{2}}}{1-\frac{q(\varphi ^{-1}(i\pi +\frac{\pi
}{2}+D))}{z^{2}}}\Big),}
\end{eqnarray}
  \item[ii)] \begin{eqnarray}\label{bb}
&&\int_{-\frac{\pi }{2}}^{-\frac{\pi
}{2}+D}{\textstyle\frac{\frac{q(\varphi ^{-1}(t+(i+1)\pi))}{z^{2}}
\Big(\sin^{2}t-t\sin t\cos t\Big)}{1-\frac{q(\varphi
^{-1}(t+(i+1)\pi))}{z^{2}}\sin^{2}t}dt}\cr&&{\textstyle
\geq\frac{\pi}{4}\Big(\frac{\frac{q(\varphi ^{-1}( i\pi +\frac{\pi
}{2}))}{z^{2}}}{1-\frac{3q(\varphi ^{-1}( i\pi+\frac{\pi}
{2}))}{4z^{2}}}\Big)+\frac{\pi}{4} \log\Big(1-\frac{q(\varphi ^{-1}(
i\pi+\frac{\pi}{2}))}{z^{2}}\Big)}\cr&&{\textstyle
-\frac{\pi}{6}\log\Big(1-\frac{3q(\varphi ^{-1}( i\pi +\frac{\pi
}{2}))}{4z^{2}}\Big),}
\end{eqnarray}
\item[iii)] there exists $c_{1}\in(i\pi
+\frac{\pi}{2},~i\pi +\frac{\pi }{2}+D)$ $(i\geq0)$ such that
\begin{eqnarray}\label{ri}
&&-(i+1)\pi\int_{-\frac{\pi }{2}}^{-\frac{\pi }{2}+D}
\frac{\frac{q(\varphi ^{-1}(t+(i+1)\pi))}{z^{2}} \sin t\cos
t}{1-\frac{q(\varphi ^{-1}(t+(i+1)\pi))}{z^{2}}\sin^{2}t}dt
\cr&&\geq-\frac{(i+1)\pi^{2}}{2z^{2}}
\Big(\frac{1}{1-\frac{q(\varphi
^{-1}(c_{1}))}{z^{2}}}\Big)\frac{q'(\varphi
^{-1}(c_{1}))}{\varphi'(\varphi ^{-1}(c_{1}))}.
\end{eqnarray}
\end{description}
\end{lemma}
\begin{proof}
\begin{description}
  \item[i)] By virtue of Lemma \ref{ty},
  \begin{eqnarray*}
\int_{\varphi ^{-1}( i\pi +\frac{\pi }{2}) }^{\varphi ^{-1}( i\pi
+\frac{\pi }{2}+D)}\frac{q(x)}{z}\sin^{2}\varphi dx=\int_{-\frac{\pi
}{2}}^{-\frac{\pi }{2}+D}\frac{\frac{q(\varphi
^{-1}(t+(i+1)\pi))}{z^{2}} \sin^{2}t}{1-\frac{q(\varphi
^{-1}(t+(i+1)\pi))}{z^{2}}\sin^{2}t} dt.
  \end{eqnarray*}
For $z^{2}>q(x_{0}),$ $\vert \frac{q(\varphi ^{-1}( t) )}{z^{2}}\sin
^{2}(t)\vert <1$, and hence, the function
$\frac{1}{1-\frac{q}{z^{2}}\sin ^{2}\varphi }$ is developable in
entire series. Thus, since $q(x)\geq 0$ is monotone increasing, then
\begin{eqnarray}\label{rit}
&&\int_{-\frac{\pi }{2}}^{-\frac{\pi }{2}+D}\frac{\frac{q(\varphi
^{-1}(t+(i+1)\pi))}{z^{2}} \sin^{2}t}{1-\frac{q(\varphi
^{-1}(t+(i+1)\pi))}{z^{2}}\sin^{2}t} dt\cr&&=\sum_{n\geq 0}
\int_{-\frac{\pi }{2}}^{-\frac{\pi }{2}+D}\Big(\frac{q(\varphi
^{-1}(t+(i+1)\pi))}{z^{2}}\Big)^{n+1}(\sin
t)^{2n+2}dt\cr&&\geq\sum_{n\geq 0}\Big(\frac{q(\varphi ^{-1}( i\pi
+\frac{\pi }{2}))}{z^{2}}\Big)^{n+1} \int_{-\frac{\pi
}{2}}^{-\frac{\pi }{2}+D}(\sin t)^{2n+2}dt.
\end{eqnarray}
By integration by parts, we obtain
\begin{eqnarray*} &&\int_{ -\frac{\pi }{2}}^{-\frac{\pi }{2}+D}(\sin t)^{2n+2}
\geq\int_{ -\frac{\pi }{2}}^{0}(\sin t)^{2n+2}
 dt=\int_{ -\frac{\pi }{2}}^{0}(\sin t)^{2n+1}\sin(t)
dt\cr&&=(2n+1)\int_{ -\frac{\pi }{2}}^{0}(\sin t)^{2n}\cos^{2}(t)
dt\cr&&=(2n+1)\int_{-\frac{\pi }{2}}^{0}(\sin t)^{2n}
dt-(2n+1)\int_{ -\frac{\pi }{2}}^{0}(\sin t)^{2n+2} dt.
\end{eqnarray*}Therefore,
 \begin{eqnarray}\label{sin}
&&\int_{ -\frac{\pi }{2}}^{-\frac{\pi }{2}+D}(\sin t)^{2n+2} dt\geq
\frac{2n+1}{2n+2}\int_{ -\frac{\pi }{2}}^{0}(\sin t)^{2n} dt\cr&&=
\prod_{p=1}^{n}\frac{2p+1}{2p+2}\int_{ -\frac{\pi
}{2}}^{0}\sin^{2}(t)dt=\frac{\pi}{4}\prod_{p=1}^{n}\frac{2p+1}{2p+2}\geq
\frac{\pi}{4}\Big(\frac{3}{4}\Big)^{n}.
\end{eqnarray}
Then
\begin{eqnarray}\label{rt}
&&\int_{\varphi ^{-1}( i\pi +\frac{\pi }{2}) }^{\varphi ^{-1}(i\pi
+\frac{\pi }{2}+D)}\frac{q(x)}{z} \sin^{2}\varphi
dx\cr&&\geq\sum_{n\geq 0}\Big(\frac{q(\varphi ^{-1}( i\pi +\frac{\pi
}{2}))}{z^{2}}\Big)^{n+1}\int_{-\frac{\pi }{2}}^{0}(\sin t)^{2n+2}
dt\cr&&\geq\frac{\pi}{4}\frac{q(\varphi ^{-1}( i\pi +\frac{\pi
}{2}))}{z^{2}}\sum_{n\geq0}\Big(\frac{3q(\varphi ^{-1}( i\pi
+\frac{\pi
}{2}))}{4z^{2}}\Big)^{n}\cr&&=\frac{\pi}{4}\Big(\frac{\frac{q(\varphi
^{-1}( i\pi +\frac{\pi }{2}))}{z^{2}}}{1-\frac{3q(\varphi ^{-1}(
i\pi +\frac{\pi }{2}))}{4z^{2}}}\Big).
\end{eqnarray}
Analogously,
\begin{eqnarray}\label{rp}
&&\int_{-\frac{\pi }{2}}^{-\frac{\pi }{2}+D}\frac{\frac{q(\varphi
^{-1}(t+(i+1)\pi))}{z^{2}} \sin^{2}t} {1-\frac{q(\varphi
^{-1}(t+(i+1)\pi))}{z^{2}}\sin^{2}t} dt\cr&&\leq\sum_{n\geq
0}{\textstyle\Big(\frac{q(\varphi ^{-1}( i\pi +\frac{\pi
}{2}+D))}{z^{2}}\Big)^{n+1}} \int_{-\frac{\pi }{2}}^{\frac{\pi
}{2}}(\sin t)^{2n+2}dt\cr&&\leq\sum_{n\geq
0}{\textstyle\Big(\frac{q(\varphi ^{-1}(i\pi +\frac{\pi
}{2}+D))}{z^{2}}\Big)^{n+1}}\int_{ -\frac{\pi }{2}}^{\frac{\pi
}{2}}\sin^{2}(t) dt\cr&&=\frac{\pi}{2}\Big(\frac{\frac{q(\varphi
^{-1}(i\pi +\frac{\pi }{2}+D))}{z^{2}}}{1-\frac{q(\varphi ^{-1}(
i\pi +\frac{\pi }{2}+D))}{z^{2}}}\Big).
\end{eqnarray} Using \eqref{rt} and \eqref{rp}, we find \eqref{aa}.
\item[ii)] Clearly, if $\vert t \vert \in ] 0,\frac{\pi }{2}[ ,$ then
$\sin ^{2}t -t \sin t \cos t
>0,$  and hence,\begin{eqnarray*}
&&\int_{-\frac{\pi }{2}}^{-\frac{\pi
}{2}+D}{\textstyle\frac{\frac{q(\varphi ^{-1}(t+(i+1)\pi))}{z^{2}}
\Big(\sin^{2}t-t\sin t\cos t\Big)}{1-\frac{q(\varphi
^{-1}(t+(i+1)\pi))}{z^{2}}\sin^{2}(t)}dt}\cr&&\geq\int_{-\frac{\pi
}{2}}^{0}{\textstyle\frac{\frac{q(\varphi ^{-1}(t+(i+1)\pi))}{z^{2}}
\Big(\sin^{2}t-t\sin t\cos t\Big)}{1-\frac{q(\varphi
^{-1}(t+(i+1)\pi))}{z^{2}}\sin^{2}(t)}dt}\cr&&=\sum_{n\geq 0}
\int_{-\frac{\pi }{2}}^{0}{\textstyle\Big(\frac{q(\varphi
^{-1}(t+(i+1)\pi))}{z^{2}}\Big)^{n+1}(\sin
t)^{2n}\Big(\sin^{2}t-t\sin t\cos t\Big)dt}\cr&&\geq\sum_{n\geq
0}{\textstyle\Big(\frac{q(\varphi ^{-1}( i\pi +\frac{\pi
}{2}))}{z^{2}}\Big)^{n+1} \int_{-\frac{\pi }{2}}^{0}(\sin
t)^{2n+2}dt}\cr&&-\sum_{n\geq 0}{\textstyle\Big(\frac{q(\varphi
^{-1}( i\pi +\frac{\pi }{2}))}{z^{2}}\Big)^{n+1} \int_{-\frac{\pi
}{2}}^{0}t(\sin t)^{2n+1}\cos(t) dt}.
\end{eqnarray*}
Integrating by parts and using \eqref{sin}, yields
\begin{eqnarray}\label{jiji}
&&-\sum_{n\geq 0}{\textstyle\Big(\frac{q(\varphi ^{-1}( i\pi
+\frac{\pi }{2}))}{z^{2}}\Big)^{n+1}} \int_{-\frac{\pi
}{2}}^{0}{\textstyle t(\sin t)^{2n+1}\cos t dt}\cr
&&=-\sum_{n\geq0}{\textstyle\Big(\frac{q(\varphi ^{-1}( i\pi
+\frac{\pi }{2}))}{z^{2}}\Big)^{n+1} \Big[\frac{t(\sin
t)^{2n+2}}{2n+2}\Big]_{-\frac{\pi}{2}}^{0}}\cr&&+
\sum_{n\geq0}{\textstyle\Big(\frac{q(\varphi ^{-1}( i\pi +\frac{\pi
}{2}))}{z^{2}}\Big)^{n+1}\frac{1}{2n+2}
\int_{-\frac{\pi}{2}}^{0}(\sin t)^{2n+2}dt}\cr&&\geq-
\frac{\pi}{4}\sum_{n\geq0}{\textstyle\frac{1}{n+1}\Big(\frac{q(\varphi
^{-1}( i\pi +\frac{\pi
}{2}))}{z^{2}}\Big)^{n+1}+\frac{\pi}{8}\sum_{n\geq0}\frac{1}{n+1}\Big(\frac{q(\varphi
^{-1}( i\pi +\frac{\pi
}{2}))}{z^{2}}\Big)^{n+1}\Big(\frac{3}{4}\Big)^{n}}\cr&&={\textstyle\frac{\pi}{4}
\log\Big(1-\frac{q(\varphi ^{-1}( i\pi +\frac{\pi
}{2}))}{z^{2}}\Big)-\frac{\pi}{6}\log\Big(1-\frac{3q(\varphi ^{-1}(
i\pi +\frac{\pi }{2}))}{4z^{2}}\Big)}.
\end{eqnarray}
Therefore, \eqref{rt} and \eqref{jiji} give \eqref{bb}.
\item[iii)] As before,
\begin{eqnarray}\label{jdid}
&&-\int_{-\frac{\pi }{2}}^{-\frac{\pi }{2}+D}{\textstyle
\frac{\frac{q(\varphi ^{-1}(t+(i+1)\pi))}{z^{2}} \sin t\cos
t}{1-\frac{q(\varphi
^{-1}(t+(i+1)\pi))}{z^{2}}\sin^{2}t}dt}\cr&&=-\int_{-\frac{\pi
}{2}}^{-\frac{\pi }{2}+D}\sum_{n\geq 0}{\textstyle
\Big(\frac{q(\varphi
^{-1}(t+(i+1)\pi))}{z^{2}}\Big)^{n+1}\sin^{2n+1}(t)\cos(t)dt}\cr&&\geq-\sum_{n\geq
0}{\textstyle\Big(\frac{q(\varphi^{-1}(i\pi +\frac{\pi
}{2}))}{z^{2}}\Big)^{n+1}}\int_{-\frac{\pi
}{2}}^{0}{\textstyle\sin^{2n+1}(t)\cos(t)dt}\cr&&-\sum_{n\geq
0}{\textstyle\Big(\frac{q(\varphi^{-1}((i+1)\pi +\frac{\pi
}{2}))}{z^{2}}\Big)^{n+1}}\int_{0}^{-\frac{\pi
}{2}+D}{\textstyle\sin^{2n+1}(t)\cos(t)dt}\cr&&=\frac{1}{2}\Big[\sum_{n\geq
0}{\textstyle\frac{1}{n+1}\Big(\frac{q(\varphi^{-1}(i\pi +\frac{\pi
}{2}))}{z^{2}}\Big)^{n+1}}-\sum_{n\geq
0}{\textstyle\frac{1}{n+1}\Big(\frac{q(\varphi^{-1}(i\pi +\frac{\pi
}{2}+D))}{z^{2}}\Big)^{n+1}\Big]}\cr&&= \frac{1}{2}
\Big[\log\Big(1-\frac{q(\varphi ^{-1}(i\pi +\frac{\pi }{2}+D)
)}{z^{2}}\Big)-\log\Big(1-\frac{q(\varphi ^{-1}( i\pi +\frac{\pi
}{2}) )}{z^{2}}\Big)\Big].
\end{eqnarray}
By the mean value theorem, there exists $c_{1}\in(i\pi +\frac{\pi
}{2}, i\pi +\frac{\pi }{2}+D)$ such that
\begin{eqnarray*}\label{21}
&&{\textstyle\Big[\log\Big(1-\frac{q(\varphi ^{-1}(i\pi +\frac{\pi
}{2}+D))}{z^{2}}\Big)-\log\Big(1-\frac{q(\varphi ^{-1}(
i\pi+\frac{\pi }{2} ))}{z^{2}}\Big)\Big]}\cr&&=-\frac{D}{z^{2}}
\Big(\frac{1}{1-\frac{q(\varphi
^{-1}(c_{1}))}{z^{2}}}\Big)\frac{\partial
q(\varphi^{-1}(c_{1}))}{\partial t}\cr&&\geq-\frac{\pi}{z^{2}}
\Big(\frac{1}{1-\frac{q(\varphi
^{-1}(c_{1}))}{z^{2}}}\Big)\frac{\partial
q(\varphi^{-1}(c_{1}))}{\partial t}\cr&&=-\frac{\pi}{z^{2}}
\Big(\frac{1}{1-\frac{q(\varphi ^{-1}(c_{1}))}{z^{2}}}\Big)
\frac{q'(\varphi ^{-1}(c_{1}))}{\varphi'(\varphi ^{-1}(c_{1}))}.
\end{eqnarray*}
Therefore, from this and \eqref{jdid}, we obtain \eqref{ri}.
\end{description}
\end{proof}
\begin{lemma} \label{3.5} Let $q(x)$ be satisfying the conditions in
Lemma \ref{wa}. Then there exists \\$c_{2}\in(i\pi +\frac{\pi
}{2},i\pi +\frac{\pi }{2}+D)$ $(\frac{\pi}{2}\leq D\leq\pi)$ such
that
\begin{eqnarray}\label{moi}
&& -2\int_{\varphi ^{-1}( i\pi +\frac{\pi }{2}) }^{\varphi ^{-1}(
i\pi +\frac{\pi }{2}+D)}\frac{q(x)}{z}\sin\varphi\cos\varphi
\Big(\int_{\varphi^{-1}(i\pi+\frac{\pi}{2})}^{x}\frac{q(s)}{z}
\sin^{2}\varphi ds
\Big)dx\cr&&\geq{\textstyle-\frac{\pi}{4}\Big[\Big(\frac{\frac{q(\varphi
^{-1}(i\pi +\frac{\pi }{2}))}{z^{2}}}{1-\frac{3}{4}\frac{q(\varphi
^{-1}(i\pi +\frac{\pi
}{2}))}{z^{2}}}\Big)\log\Big(1-\frac{3}{4}\frac{q(\varphi ^{-1}(i\pi
+\frac{\pi }{2}))}{z^{2}}\Big)-2\Big(\frac{\frac{q(\varphi
^{-1}(i\pi +\frac{\pi }{2}))}{z^{2}}}{1-\frac{q(\varphi ^{-1}(i\pi
+\frac{\pi }{2}))}{z^{2}}}\Big)\log\Big(1-\frac{q(\varphi ^{-1}(i\pi
+\frac{\pi }{2}))}{z^{2}}\Big)\Big]}\cr&&+\frac{\pi^{2}}{2z^{2}}
\Big(\frac{\log(1-\frac{q(\varphi^{-1}(c_{2}))}{z^{2}})
-\frac{q(\varphi^{-1}(c_{2}))}{z^{2}}}
{(1-\frac{q(\varphi^{-1}(c_{2}))}{z^{2}})^{2}}\Big)\frac{q'(\varphi
^{-1}(c_{2}))}{\varphi'(\varphi ^{-1}(c_{2}))}.
\end{eqnarray}
\end{lemma}
\begin{proof}
For $z^{2}>q(x_{0}),$ we have
\begin{eqnarray*}
&& -2\int_{\varphi ^{-1}( i\pi +\frac{\pi }{2}) }^{\varphi ^{-1}(
i\pi+\frac{\pi}{2}+D)}{\textstyle\frac{q(x)}{z}\sin\varphi\cos\varphi}
\Big(\int_{\varphi^{-1}(i\pi+\frac{\pi}{2})}^{x}{\textstyle\frac{q(s)}{z}
\sin^{2}\varphi ds \Big)dx}\cr&&=-2\int_{\varphi ^{-1}( i\pi
+\frac{\pi }{2}) }^{\varphi ^{-1}(i\pi +\frac{\pi
}{2}+D)}\frac{\frac{q(x)}{z}\varphi'\sin\varphi\cos\varphi}
{z-\frac{q(x)}{z}\sin^{2}\varphi}
\Big(\int_{\varphi^{-1}(i\pi+\frac{\pi}{2})}^{x}\frac{q(s)}{z}
\sin^{2}\varphi ds \Big)dx\cr&&=-2\sum_{n\geq 0}\int_{\varphi ^{-1}(
i\pi +\frac{\pi }{2}) }^{\varphi
^{-1}(i\pi+\frac{\pi}{2}+D)}{\textstyle\Big(\frac{q(x)}{z^{2}}\Big)^{n+1}\varphi'\sin^{2n+1}(\varphi)
\cos(\varphi)}
\Big(\int_{\varphi^{-1}(i\pi+\frac{\pi}{2})}^{x}{\textstyle\frac{q(s)}{z}
\sin^{2}\varphi ds \Big)dx}.
\end{eqnarray*}
Integrating by parts and taking into account that $q(x)$ be monotone
increasing on $[0,x_0]$, we find
\begin{eqnarray*}
&&-2\sum_{n\geq 0}\int_{\varphi ^{-1}( i\pi +\frac{\pi }{2})
}^{\varphi
^{-1}(i\pi+\frac{\pi}{2}+D)}{\textstyle\Big(\frac{q(x)}{z^{2}}\Big)^{n+1}
\varphi'\sin^{2n+1}(\varphi)\cos(\varphi)}
\Big(\int_{\varphi^{-1}(i\pi+\frac{\pi}{2})}^{x}{\textstyle\frac{q(s)}{z}
\sin^{2}\varphi ds \Big)dx}\cr&&=-2\sum_{n\geq
0}\Big[\frac{\sin^{2n+2}(\varphi)}{2n+2}\Big(\frac{q(x)}{z^{2}}\Big)^{n+1}
\int_{\varphi^{-1}(i\pi+\frac{\pi}{2})}^{x}\frac{q(s)}{z}
\sin^{2}\varphi ds\Big]_{\varphi ^{-1}( i\pi +\frac{\pi }{2})
}^{\varphi ^{-1}(i\pi +\frac{\pi }{2}+D)}\cr&&+\sum_{n\geq
0}\frac{1}{n+1}\int_{\varphi ^{-1}( i\pi +\frac{\pi }{2}) }^{\varphi
^{-1}(i\pi+\frac{\pi}{2}+D)}{\textstyle\sin^{2n+4}
(\varphi)\Big(\frac{q(x)}{z^{2}}\Big)^{n+1}\frac{q(x)}{z}dx}
\cr&&+\sum_{n\geq 0}\int_{\varphi ^{-1}( i\pi +\frac{\pi }{2})
}^{\varphi ^{-1}(i\pi +\frac{\pi
}{2}+D)}{\textstyle\Big(\frac{q(x)}{z^{2}}\Big)^{n}
\frac{q'(x)}{z^{2}}\sin^{2n+2}(\varphi)
}\Big(\int_{\varphi^{-1}(i\pi+\frac{\pi}{2})}^{x}{\textstyle\frac{q(s)}{z}
\sin^{2}\varphi ds \Big)dx}\cr&&\geq -\sum_{n\geq
0}{\textstyle\Big(\frac{q(\varphi ^{-1}(i\pi +\frac{\pi
}{2}+D))}{z^{2}}\Big)^{n+1}\frac{1}{n+1}}
\Big(\int_{\varphi^{-1}(i\pi+\frac{\pi}{2})}^{\varphi ^{-1}(i\pi
+\frac{\pi }{2}+D)}{\textstyle\frac{q(s)}{z} \sin^{2}(\varphi)
ds\Big)}\cr&&+\sum_{n\geq 0}\frac{1}{n+1}\int_{\varphi ^{-1}( i\pi
+\frac{\pi }{2}) }^{\varphi ^{-1}(i\pi +\frac{\pi
}{2}+D)}{\textstyle\Big(\frac{q(x)}{z^{2}}\Big)^{n+1}
\frac{q(x)}{z}\sin^{2n+4}(\varphi)dx} \cr&&\geq-\Big[\sum_{n\geq
0}{\textstyle\frac{1}{n+1}\Big(\frac{q(\varphi ^{-1}(i\pi +\frac{\pi
}{2}+D))}{z^{2}}\Big)^{n+1}\Big]}\int_{\varphi ^{-1}( i\pi
+\frac{\pi }{2}) }^{\varphi ^{-1}(i\pi +\frac{\pi
}{2}+D)}{\textstyle\frac{q(x)}{z}\sin^{2}\varphi
dx}\cr&&+\sum_{n\geq 0}{\textstyle\frac{1}{n+1} \Big(\frac{q(\varphi
^{-1}(i\pi +\frac{\pi }{2}))}{z^{2}}\Big)^{n+1}}\int_{\varphi ^{-1}(
i\pi +\frac{\pi }{2}) }^{\varphi ^{-1}(i\pi +\frac{\pi
}{2}+D)}{\textstyle\frac{q(x)}{z}\sin^{2n+4}\varphi dx}.
\end{eqnarray*}
By \eqref{aa},
\begin{eqnarray}\label{z}&&-\Big[\sum_{n\geq
0}{\textstyle\frac{1}{n+1}\Big(\frac{q(\varphi ^{-1}(i\pi +\frac{\pi
}{2}+D))}{z^{2}}\Big)^{n+1}}\Big]\int_{\varphi ^{-1}( i\pi
+\frac{\pi }{2}) }^{\varphi ^{-1}(i\pi +\frac{\pi
}{2}+D)}{\textstyle\frac{q(x)}{z}\sin^{2}\varphi
dx}\cr&&\geq{\textstyle\frac{\pi}{2}\Big(\frac{\frac{q(\varphi
^{-1}(i\pi +\frac{\pi }{2}+D))}{z^{2}}}{1-\frac{q(\varphi ^{-1}(i\pi
+\frac{\pi }{2}+D))}{z^{2}}}\Big)\log\Big(1-\frac{q(\varphi
^{-1}(i\pi +\frac{\pi }{2}+D))}{z^{2}}\Big)}.
\end{eqnarray}
On the other hand, using \eqref{sin},
\begin{eqnarray*}
&&\int_{\varphi ^{-1}( i\pi +\frac{\pi }{2}) }^{\varphi ^{-1}(i\pi
+\frac{\pi }{2}+D)}\frac{q(x)}{z}\sin^{2n+4}\varphi
dx\cr&&=\sum_{p\geq 0}\int_{ i\pi +\frac{\pi }{2}}^{i\pi +\frac{\pi
}{2}+D}{\textstyle\Big(\frac{q(\varphi
^{-1}(t))}{z^{2}}\Big)^{p+1}\sin^{2(n+p)+4}(t)dt}\cr&&\geq\sum_{p\geq
0}{\textstyle\Big(\frac{q(\varphi ^{-1}( i\pi +\frac{\pi
}{2}))}{z^{2}}\Big)^{p+1}}\int_{ i\pi +\frac{\pi }{2}}^{i\pi
+\frac{\pi}{2}+D}{\textstyle\sin^{2(n+p+1)+2}(t)dt}
\cr&&\geq\frac{\pi}{4}\Big(\frac{3}{4}\Big)^{n+1}
\frac{q(\varphi^{-1}( i\pi +\frac{\pi }{2}))}{z^{2}}\sum_{p\geq
0}{\textstyle\Big(\frac{3q(\varphi ^{-1}( i\pi +\frac{\pi
}{2}))}{4z^{2}}\Big)^{p}}\cr&&
\geq\frac{\pi}{4}\Big(\frac{3}{4}\Big)^{n+1}{\textstyle\Big(\frac{\frac{q(\varphi
^{-1}(i\pi +\frac{\pi }{2}))}{z^{2}}}{1-\frac{3}{4}\frac{q(\varphi
^{-1}(i\pi +\frac{\pi }{2}))}{z^{2}}}\Big)}.
\end{eqnarray*}
Thus
\begin{eqnarray}\label{tn}
&&\sum_{n\geq 0}{\textstyle\frac{1}{n+1} \Big(\frac{q(\varphi
^{-1}(i\pi +\frac{\pi }{2}))}{z^{2}}\Big)^{n+1}}\int_{\varphi ^{-1}(
i\pi +\frac{\pi }{2}) }^{\varphi ^{-1}(i\pi +\frac{\pi
}{2}+D)}{\textstyle\frac{q(x)}{z}\sin^{2n+4}\varphi
dx}\cr&&\geq{\textstyle-\frac{\pi}{4}\Big(\frac{\frac{q(\varphi
^{-1}(i\pi +\frac{\pi }{2}))}{z^{2}}}{1-\frac{3}{4}\frac{q(\varphi
^{-1}(i\pi +\frac{\pi
}{2}))}{z^{2}}}\Big)\log\Big(1-\frac{3}{4}\frac{q(\varphi ^{-1}(i\pi
+\frac{\pi }{2}))}{z^{2}}\Big)}.
\end{eqnarray}
 Therefore, by \eqref{z} and \eqref{tn},
 \begin{eqnarray}\label{jrim}
&&-2\int_{\varphi ^{-1}( i\pi +\frac{\pi }{2}) }^{\varphi ^{-1}(
i\pi+\frac{\pi}{2}+D)}{\textstyle\frac{q(x)}{z}\sin\varphi\cos\varphi}
\Big(\int_{\varphi^{-1}(i\pi+\frac{\pi}{2})}^{x}{\textstyle\frac{q(s)}{z}
\sin^{2}\varphi ds
\Big)dx}\cr&&{\textstyle\geq\frac{\pi}{4}\Big[2\Big(\frac{\frac{q(\varphi
^{-1}(i\pi +\frac{\pi }{2}+D))}{z^{2}}}{1-\frac{q(\varphi ^{-1}(i\pi
+\frac{\pi }{2}+D))}{z^{2}}}\Big)\log\Big(1-\frac{q(\varphi
^{-1}(i\pi +\frac{\pi
}{2}+D))}{z^{2}}\Big)-\Big(\frac{\frac{q(\varphi ^{-1}(i\pi
+\frac{\pi }{2}))}{z^{2}}}{1-\frac{3}{4}\frac{q(\varphi ^{-1}(i\pi
+\frac{\pi }{2}))}{z^{2}}}\Big)\log\Big(1-\frac{3}{4}\frac{q(\varphi
^{-1}(i\pi +\frac{\pi
}{2}))}{z^{2}}\Big)\Big]}\cr&&{\textstyle=\frac{\pi}{2}\Big[\Big(\frac{\frac{q(\varphi
^{-1}(i\pi +\frac{\pi }{2}+D))}{z^{2}}}{1-\frac{q(\varphi ^{-1}(i\pi
+\frac{\pi }{2}+D))}{z^{2}}}\Big)\log\Big(1-\frac{q(\varphi
^{-1}(i\pi +\frac{\pi
}{2}+D))}{z^{2}}\Big)-\Big(\frac{\frac{q(\varphi ^{-1}(i\pi
+\frac{\pi }{2}))}{z^{2}}}{1-\frac{q(\varphi ^{-1}(i\pi +\frac{\pi
}{2}))}{z^{2}}}\Big)\log\Big(1-\frac{q(\varphi ^{-1}(i\pi +\frac{\pi
}{2}))}{z^{2}}\Big)\Big]}\cr&&{\textstyle-\frac{\pi}{4}\Big(\frac{\frac{q(\varphi
^{-1}(i\pi +\frac{\pi }{2}))}{z^{2}}}{1-\frac{3}{4}\frac{q(\varphi
^{-1}(i\pi +\frac{\pi
}{2}))}{z^{2}}}\Big)\log\Big(1-\frac{3}{4}\frac{q(\varphi ^{-1}(i\pi
+\frac{\pi
}{2}))}{z^{2}}\Big)}\cr&&{\textstyle-\frac{\pi}{2}\Big(\frac{\frac{q(\varphi
^{-1}(i\pi +\frac{\pi }{2}))}{z^{2}}}{1-\frac{q(\varphi ^{-1}(i\pi
+\frac{\pi }{2}))}{z^{2}}}\Big)\log\Big(1-\frac{q(\varphi ^{-1}(i\pi
+\frac{\pi }{2}))}{z^{2}}\Big)}.
\end{eqnarray}
 By the mean value theorem, there exists $c_{2}\in(i\pi +\frac{\pi
}{2},i\pi +\frac{\pi }{2}+D)$ such that
\begin{eqnarray}\label{jd}
&&{\textstyle\Big(\frac{\frac{q(\varphi ^{-1}(i\pi +\frac{\pi
}{2}+D))}{z^{2}}}{1-\frac{q(\varphi ^{-1}(i\pi +\frac{\pi
}{2}+D))}{z^{2}}}\Big)\log\Big(1-\frac{q(\varphi ^{-1}(i\pi
+\frac{\pi }{2}+D))}{z^{2}}\Big)-\Big(\frac{\frac{q(\varphi
^{-1}(i\pi +\frac{\pi }{2}))}{z^{2}}}{1-\frac{q(\varphi ^{-1}(i\pi
+\frac{\pi }{2}))}{z^{2}}}\Big)\log\Big(1-\frac{q(\varphi ^{-1}(i\pi
+\frac{\pi }{2}))}{z^{2}}\Big)}\cr&&=\frac{D}{z^{2}}
\Big(\frac{\log(1-\frac{q(\varphi^{-1}(c_{2}))}{z^{2}})-
\frac{q(\varphi^{-1}(c_{2}))}{z^{2}}}
{(1-\frac{q(\varphi^{-1}(c_{2}))}{z^{2}})^{2}}\Big)\frac{\partial
q(\varphi^{-1}(c_{2}))}{\partial t}\cr&&\geq\frac{\pi}{z^{2}}
\Big(\frac{\log(1-\frac{q(\varphi^{-1}(c_{2}))}{z^{2}})-
\frac{q(\varphi^{-1}(c_{2}))}{z^{2}}}
{(1-\frac{q(\varphi^{-1}(c_{2}))}{z^{2}})^{2}}\Big)\frac{\partial
q(\varphi^{-1}(c_{2}))}{\partial t}\cr&&=\frac{\pi}{z^{2}}
\Big(\frac{\log(1-\frac{q(\varphi^{-1}(c_{2}))}{z^{2}})-
\frac{q(\varphi^{-1}(c_{2}))}{z^{2}}}
{(1-\frac{q(\varphi^{-1}(c_{2}))}{z^{2}})^{2}}\Big) \frac{q'(\varphi
^{-1}(c_{2}))}{\varphi'(\varphi ^{-1}(c_{2}))}.
\end{eqnarray}
From this and  \eqref{jrim}, we get \eqref{moi}.
\end{proof}
We are now ready to prove Theorem \ref{the1}.
\begin{proof} If for some $\lambda>0$, $y'(x,z)$ has no zeros
in $(0,1)$, then $\varphi(x,z)<\frac{\pi}{2}$ for $x\in(0,x_0)$, and
hence
\begin{eqnarray}\label{xc}
\int_{0}^{x_0}r^{2}(x) \frac{q(x)}{z}\Big( \sin^{2}\varphi
-\varphi\sin\varphi\cos\varphi\Big) dx\geq0.
\end{eqnarray}
In the rest of the proof let $\varphi ( x_{0},z) =k\pi +\frac{\pi
}{2}+D,$ with $k\geq0$ be an integer and $0\leq D\leq \pi .$ Then
\begin{eqnarray}\label{x_0}
&&\int_{0}^{x_{0}}r^{2}(x)\frac{q(x)}{z}\Big( \sin^{2}\varphi
-\varphi\sin\varphi\cos\varphi\Big) dx \cr &=&\int_{0}^{\varphi
^{-1}( k\pi +\frac{\pi }{2}+D)}r^{2}(x) \frac{q(x)}{z}\Big(
\sin^{2}\varphi -\varphi\sin\varphi\cos\varphi\Big) dx
\cr&=&\int_{0}^{\varphi ^{-1}( \frac{\pi }{2}) }\tau(x)
dx+\sum_{i=0}^{k-1}\int_{\varphi ^{-1}( i\pi +\frac{\pi }{2})
}^{\varphi ^{-1}((i+1) \pi +\frac{\pi }{2}) }\tau(x)
dx+\int_{\varphi ^{-1}( k\pi +\frac{\pi }{2}) }^{\varphi ^{-1}( k\pi
+\frac{\pi }{2}+D) }\tau(x) dx,
\end{eqnarray}
where $\tau(x)= r^{2}(x)\frac{q(x)}{z}\Big( \sin^{2}\varphi
-\varphi\sin\varphi\cos\varphi\Big).$ Since $\varphi\in]0,\frac
{\pi}{2}],$ then
\begin{eqnarray}\label{rr}
\int_{0}^{\varphi ^{-1}(\frac{\pi }{2})}r^{2}(x) \frac{q(x)}{z}\Big(
\sin^{2}\varphi -\varphi\sin\varphi\cos\varphi\Big) dx\geq0.
\end{eqnarray}
It is easily seen that if $0\leq D\leq \frac{\pi }{2}$, then
\begin{eqnarray}\label{23}
\int_{\varphi ^{-1}(i\pi +\frac{\pi }{2})}^{\varphi ^{-1}(i\pi
+\frac{\pi }{2}+D)}r^{2}(x)\frac{q(x)}{z}\Big(\sin^{2}\varphi
-\varphi\sin \varphi \cos \varphi \Big)dx \geq 0.
\end{eqnarray}%
If $\frac{\pi }{2}\leq D\leq \pi $, then in view of Lemma \ref{ty},
together with \eqref{bb} and \eqref{ri},
\begin{eqnarray}\label{aj}
&&\int_{\varphi ^{-1}(i\pi +\frac{\pi }{2}) }^{\varphi ^{-1}(i\pi
+\frac{\pi
}{2}+D)}{\textstyle\frac{q(x)}{z}\Big(\sin^{2}\varphi-\varphi\sin
\varphi \cos \varphi
\Big)dx}\cr&&\geq{\textstyle\frac{\pi}{4}\frac{\frac{q(\varphi
^{-1}( i\pi +\frac{\pi }{2}))}{z^{2}}}{1-\frac{3q(\varphi ^{-1}(
i\pi+\frac{\pi} {2}))}{4z^{2}}}+\frac{\pi}{4}
\log\Big(1-\frac{q(\varphi ^{-1}( i\pi+\frac{\pi}{2}))}{z^{2}}\Big)
-\frac{\pi}{6}\log\Big(1-\frac{3q(\varphi ^{-1}( i\pi +\frac{\pi
}{2}))}{4z^{2}}\Big)}\cr&&-\frac{(i+1)\pi^{2}}{2z^{2}}
\Big(\frac{1}{1-\frac{q(\varphi
^{-1}(c_{1}))}{z^{2}}}\Big)\frac{q'(\varphi
^{-1}(c_{1}))}{\varphi'(\varphi ^{-1}(c_{1}))}.
\end{eqnarray}
Under the hypotheses $z^2\geq 11q(x_{0})$, $\varphi'(\varphi
^{-1}(c_{1}))\geq\frac{10z}{11}$ and $\frac{1}{1-\frac{q(\varphi
^{-1}(c_{1}))}{z^{2}}}\leq\frac{11}{10}$. On the other hand, recall
that if $x_0=\varphi^{-1}(i\pi+\frac{\pi}{2}+D)$ $(\frac{\pi}{2}\leq
D\leq\pi),$ then the solution $y(x,z)$ of \eqref{1.1}-\eqref{1.2}
has exactly $(i+1)$ zeros in $(0,x_0].$ In view of Sturm oscillation
theorem, we have necessarily $z\geq z_{i+1},$ where
$\lambda_{i}=z_{i}^{2}$ are the eigenvalues of Problem
\eqref{1.1}-\eqref{1.2}. Since $z_{i+1}\geq (i+1)\pi$, then
\begin{eqnarray}\label{osc}\frac{(i+1)\pi}{z}\leq1,~~i\geq0.\end{eqnarray}
Therefore, by \eqref{aj} and \eqref{osc},
\begin{eqnarray} \label{jjj}
&&\int_{\varphi ^{-1}( i\pi +\frac{\pi }{2}) }^{\varphi ^{-1}( i\pi
+\frac{3\pi
}{2})}{\textstyle\frac{q(x)}{z}\Big(\sin^{2}\varphi-\varphi\sin
\varphi \cos \varphi
\Big)dx}\cr&&\geq{\textstyle\frac{\pi}{4}\frac{\frac{q(\varphi
^{-1}( i\pi +\frac{\pi }{2}))}{z^{2}}}{1-\frac{3q(\varphi ^{-1}(
i\pi+\frac{\pi} {2}))}{4z^{2}}}+\frac{\pi}{4}
\log\Big(1-\frac{q(\varphi ^{-1}( i\pi+\frac{\pi}{2}))}{z^{2}}\Big)
-\frac{\pi}{6}\log\Big(1-\frac{3q(\varphi ^{-1}( i\pi +\frac{\pi
}{2}))}{4z^{2}}\Big)}\cr&&-\frac{121\pi}{200}\frac{q'(\varphi
^{-1}(c_{1}))}{z^{2}}.
\end{eqnarray} On the other hand, it can be easily verified that
$\frac{\log(1-s)-s}{(1-s)^{2}}\geq \frac{-1}{4}$ for
$s\in[0,\frac{1}{11}]$. Using this, together with \eqref{moi} and
\eqref{osc} (for $i=0$), we obtain

\begin{eqnarray}\label{moi2}
&& -2\int_{\varphi ^{-1}( i\pi +\frac{\pi }{2}) }^{\varphi ^{-1}(
i\pi +\frac{3\pi }{2})}\frac{q(x)}{z}\sin\varphi\cos\varphi
\Big(\int_{\varphi^{-1}(i\pi+\frac{\pi}{2})}^{x}\frac{q(s)}{z}
\sin^{2}\varphi ds
\Big)dx\cr&&\geq{\textstyle-\frac{\pi}{4}\Big[\Big(\frac{\frac{q(\varphi
^{-1}(i\pi +\frac{\pi }{2}))}{z^{2}}}{1-\frac{3}{4}\frac{q(\varphi
^{-1}(i\pi +\frac{\pi
}{2}))}{z^{2}}}\Big)\log\Big(1-\frac{3}{4}\frac{q(\varphi ^{-1}(i\pi
+\frac{\pi }{2}))}{z^{2}}\Big)-2\Big(\frac{\frac{q(\varphi
^{-1}(i\pi +\frac{\pi }{2}))}{z^{2}}}{1-\frac{q(\varphi ^{-1}(i\pi
+\frac{\pi }{2}))}{z^{2}}}\Big)\log\Big(1-\frac{q(\varphi ^{-1}(i\pi
+\frac{\pi }{2}))}{z^{2}}\Big)\Big]}\cr&&-\frac{11\pi}{80}
\frac{q'(\varphi ^{-1}(c_{2}))}{z^{2}}.
\end{eqnarray}
Thus, Summing \eqref{jjj} and \eqref{moi2}, one gets
\begin{eqnarray}\label{fina}
&&\int_{\varphi ^{-1}( i\pi +\frac{\pi }{2}) }^{\varphi ^{-1}( i\pi
+\frac{3\pi }{2})}\frac{q(x)}{z} \Big(\sin^{2}\varphi-\varphi\sin
\varphi \cos \varphi \Big)dx \cr&&-2\int_{\varphi ^{-1}( i\pi
+\frac{\pi }{2}) }^{\varphi ^{-1}( i\pi +\frac{3\pi
}{2})}\frac{q(x)}{z}\sin \varphi \cos \varphi
\Big(\int_{\varphi^{-1}(i\pi+\frac{\pi}{2})}^{x}\frac{q(s)}{z}
\sin^{2}\varphi ds\Big)dx\cr&&\geq \pi G(\frac{q(\varphi ^{-1}(i\pi
+\frac{\pi}{2}))}{z^{2}})-\frac{297\pi}{400z^{2}}q'(\varphi^{-1}(c))
\cr&&\geq \pi G(\frac{q(\varphi ^{-1}(i\pi
+\frac{\pi}{2}))}{z^{2}})-\frac{3\pi}{4z^{2}}q'(\varphi^{-1}(c)),
\end{eqnarray}
where ${\textstyle G(s)=\frac{s}{4(1-\frac{3}{4}s)}\Big(1-\log
(1-\frac{3}{4}s)\Big)+\frac{1}{4}\frac{(1+s)\log
(1-s)}{4(1-s)}-\frac{1}{6}\log (1-\frac{3}{4}s)}$ and\\
$q'(\varphi^{-1}(c))=\max\{q'(\varphi^{-1}(c_{1})),q'(\varphi^{-1}(c_{2}))\}.$
It can be shown by straightforward computation that
$G(s)\geq\frac{s}{10}$ for $s\in[0,\frac{1}{11}].$ By setting
$s=\frac{q(\varphi ^{-1}(i\pi +\frac{\pi}{2}))}{z^{2}}\geq
\frac{q(0)}{z^2}$ and taking into account the condition
$q'(x)\leq\frac{2}{15}q(0)$, we find out that
\begin{eqnarray}\label{JIRIM}
&&\int_{\varphi ^{-1}( i\pi +\frac{\pi }{2}) }^{\varphi ^{-1}( i\pi
+\frac{\pi }{2}+D)}\frac{q(x)}{z} \Big(\sin^{2}\varphi-\varphi\sin
\varphi \cos \varphi \Big)dx \cr&&-2\int_{\varphi ^{-1}( i\pi
+\frac{\pi }{2}) }^{\varphi ^{-1}( i\pi +\frac{\pi
}{2})}\frac{q(x)}{z}\sin \varphi \cos \varphi
\Big(\int_{\varphi^{-1}(i\pi+\frac{\pi}{2})}^{x}\frac{q(s)}{z}
\sin^{2}\varphi ds\Big)dx\cr&&\geq 0.
\end{eqnarray}
Now, according to Lemma \ref{Lem4} together with \eqref{JIRIM}, we
conclude that
\begin{eqnarray}\label{cc}
\int_{\varphi ^{-1}(k\pi +\frac{\pi }{2})}^{\varphi ^{-1}( k\pi
+\frac{\pi }{2}+D)}r^{2}(x)\frac{q(x)}{z}\Big(\sin^{2}\varphi
-\varphi\sin \varphi \cos \varphi \Big)dx \geq 0.
\end{eqnarray}
Therefore, by \eqref{x_0}, \eqref{rr}, \eqref{23} and \eqref{cc} we
have $\dot{\theta}( x_{0},z) \geq 0 $ for $z^{2}\geq 11q(x_{0})$ and
$q'(x)\leq \frac{2}{15}q(0).$ Obviously, if there is a
$z\geq\sqrt{11q(x_{0})} $ with $\dot{\theta}(x_{0},z) =0,$ then in
view of \eqref{x_0},
\begin{eqnarray*}
\int_{0}^{\varphi ^{-1}(\frac{\pi }{2})}r^{2}(x) \frac{q(x)}{z}\Big(
\sin^{2}\varphi -\varphi\sin\varphi\cos\varphi\Big) dx=0
\end{eqnarray*}
which is not possible unless $q\equiv0$ in $[0,x_{0}]$. The theorem
is proved.
\end{proof}
\section{Proof of Theorem \ref{the2}} \label{4}
Let the potential $q(x)$ be monotone increasing in $[0,x_{0}]$ and
monotone decreasing in $[x_{0},1].$ We denote by $\tilde{q}(x)$ the
reverse of the potential, i.e., $\tilde{q}(x)=q(1-x)$. Then
$\tilde{y}( x,z) $ is the solution of the initial value problem
\begin{eqnarray} \label{22.1} \left\{
\begin{array}{ll}{-\tilde{y}}^{\prime \prime }{( x,z) +\tilde{q}(x)\tilde{y}( x,z)
=z}^{2}{\tilde{y}}(x,z),\\
\tilde{y}( 0) =0,~~ \tilde{y}^{\prime }( 0) =1.
 \end{array}
 \right.
\end{eqnarray}
As in section \ref{ps}, we define the associated Pr\"{u}fer transformation%
\begin{eqnarray} \label{22.2} \left\{
\begin{array}{ll}\tilde{y}( x,z)=\tilde{r}( x,z) \sin (
z\tilde{\theta}( x,z) ) , \\
\tilde{y}^{\prime }( x,z) =z\tilde{r}( x,z) \cos (
z\tilde{\theta}( x,z) ),\\
\tilde{\theta}( 0,z) =0.
 \end{array}
 \right.
\end{eqnarray}
As in \cite{3}, we have the following relations:
\begin{eqnarray} \label{3.2}
\left\{
 \begin{array}{ll} \tilde{y}( x,z_{n}) =( -1) ^{n+1}\frac{y(1-x,z_{n})}{%
z_{n}r( 1,z_{n}) },\\
\tilde{r}( x,z_{n}) =\frac{r(1-x,z_{n})}{r( 1,z_{n}) },\\
\tilde{\theta}( x,z_{n}) =\frac{n\pi }{z_{n}}-\theta (1-x,z_{n}),
 \end{array}
 \right.
\end{eqnarray}
where $\lambda_{n}=z_{n}^{2}$ is an eigenvalue of Problem
\eqref{1.1}-\eqref{1.2}.\\
\begin{proof}{ of Theorem \ref{the2}.}
As $\tilde{q}(x)=q(1-x),$ then $\tilde{q}(x)$\ be monotone
increasing in $[0,1-x_{0}]$ and monotone decreasing in
$[1-x_{0},1]$. We have $\tilde{q}'(x)=-q'(1-x)\leq
\frac{2}{15}\tilde{q}(0)=\frac{2}{15}q(1)$, thus for $x\in[0,1]$,
$|q'(x)|\leq q^{*}$ where $q^{*}=\frac{2}{15}\min\{q(0),~q(1)\}$.
Therefore by Theorem \ref{the1}, $\tilde{\theta}( 1-x_{0},z) $ is
increasing for $z\geq \sqrt{11\tilde{q}(1-x_{0})}=\sqrt{11q(
x_{0})}.$ Consequently, the function $\Psi ( z) =\theta ( x_{0},z)
+\tilde{\theta}( 1-x_{0},z) $ is increasing for $z\geq \sqrt{11q(
x_{0})}$. By \eqref{3.2}, $z_{n}\Psi ( z_{n}) =n\pi.$ Let $m$ be an
integer such that $m<n$ and $\lambda_{m}\geq11q(x_{0})$. Then
$$\Psi ( z_{n}) =\frac{n\pi }{z_{n}}\geq \Psi ( z_{m}) =%
\frac{m\pi }{z_{m}},$$ which implies that $\frac{\lambda
_{n}}{\lambda _{m}}\leq \frac{n^{2}}{m^{2}}.$ On the other hand, if
$q(x_{0})\leq\frac{\pi^{2}}{11},$ then
$$z_{1}=\sqrt{\lambda_{1}}\geq \sqrt{\pi^{2}+q^{-}} \geq \sqrt{11q(
x_{0})}$$ where $q^{-}=\min_{x\in[0,1]} q(x).$ Thus, in this case $
\frac{\lambda _{n}}{\lambda _{m}}\leq \frac{n^{2}}{m^{2}}$ for all
$n>m\geq1.$ If equality holds, then $\Psi ( z_{n}) =\Psi ( z_{m}),$
so that $\dot{\Psi}(z)=0$ for some $z>0.$ Hence
$\dot{\theta}(x_{0},z)=\dot{\tilde{\theta}}( 1-x_{0},z) =0 ,$ and in
view of Theorem \ref{the1}, $q\equiv0$ in $[ 0,x_{0}] $ and
$\tilde{q}\equiv0$ in $[ 0,1-x_{0}],$ i.e., $q\equiv0$ in $[0,1].$
\end{proof}

\end{document}